\documentclass[11pt]{article}
\usepackage{amsthm,amssymb,amsmath}
\usepackage{graphicx,algorithmic,algorithm}%
\usepackage{enumerate}
\usepackage{tikz,tikz-cd}
\tikzcdset{scale cd/.style={every label/.append style={scale=#1},
    cells={nodes={scale=#1}}}}%
\usetikzlibrary{shapes}
\usepackage[letterpaper, margin=1in]{geometry}
\usepackage{hyperref}
\usepackage{todonotes}
\usepackage[normalem]{ulem}
\hypersetup{
	pdftitle=   {Partition-width},
	pdfauthor=  {Yeonsu Chang, O-joung Kwon and Myounghwan Lee}
}
\usepackage{authblk}
\usepackage{mathabx}
\usepackage{caption}
\usepackage{subcaption}

\newtheorem{theorem}{Theorem}[section]
\newtheorem{lemma}[theorem]{Lemma}

\newtheorem*{thm:main1}{Theorem~\ref{thm:main1}}
\newtheorem*{thm:main2}{Theorem~\ref{thm:main2}}

\newcommand\extrafootertext[1]{%
    \bgroup
    \renewcommand\thefootnote{\fnsymbol{footnote}}%
    \renewcommand\thempfootnote{\fnsymbol{mpfootnote}}%
    \footnotetext[0]{#1}%
    \egroup
}

\newcommand\fw{\operatorname{fw}}

\newcommand\abs[1]{\lvert #1\rvert}

\newcommand{\obs}{\operatorname{obs}}

\title{A characterization of graphs of radius-$r$ flip-width at most $2$}
	\author{Yeonsu Chang}
    \author{Sejin Ko}
	\author[1,2]{O-joung Kwon}
	\author{Myounghwan Lee}

	\affil{Department of Mathematics, Hanyang University, Seoul, South Korea.}
	\affil[2]{Discrete Mathematics Group, Institute for Basic Science (IBS), Daejeon, South Korea}
	
	\date\today

\begin{document}

\maketitle
\extrafootertext{Y. Chang, O. Kwon, and M. Lee are supported by the National Research Foundation of Korea (NRF) grant funded by the Ministry of Science and ICT (No. NRF-2021K2A9A2A11101617 and No. RS-2023-00211670). S. Ko is supported by the National Research Foundation of Korea (NRF) grant funded by the Korea government MSIT (No. NRF-2022R1C1C1010300). O. Kwon is also supported  by Institute for Basic Science (IBS-R029-C1).}

	\extrafootertext{E-mail addresses: \texttt{yeonsu@hanyang.ac.kr} (Y. Chang), 
   \texttt{tpwls960104@hanyang.ac.kr} (S. Ko),  \texttt{ojoungkwon@hanyang.ac.kr} (O. Kwon) and \texttt{sycuel@hanyang.ac.kr} (M. Lee) }

\begin{abstract}
    The radius-$r$ flip-width of a graph, for $r\in \mathbb{N}\cup \{\infty\}$, is a graph parameter defined in terms of a variant of the cops and robber game, called the flipper game, and it was introduced by Toru\'{n}czyk (FOCS 2023). 
    We prove that for every $r\in (\mathbb{N}\setminus \{1\})\cup \{\infty\}$, the class of graphs of radius-$r$ flip-width at most $2$ is exactly the class of ($C_5$, bull, gem, co-gem)-free graphs, which are known as totally decomposable graphs with respect to bi-joins. 
\end{abstract}
\section{Introduction}

Toru\'{n}czyk~\cite{flipwidth} introduced a graph parameter called radius-$r$ flip-width for $r\in \mathbb{N}\cup \{\infty\}$, which is defined using a variant of the cops and robber game, so called the flipper game. In the flipper game, instead of positioning cops on vertices, the flipper announces a partition of the vertex set and flips edges on some parts, or between some two parts, and the runner can move from the previous position to a new position by traversing a path of length at most $r$. For $r\in \mathbb{N}\cup \{\infty\}$, the \emph{radius-$r$ flip-width} of a graph is the minimum~$k$ for which the flipper has a strategy to catch the runner using a sequence of such flippings with partitions into at most $k$ parts. A class $\mathcal{C}$ of graphs has \emph{bounded flip-width} if the radius-$r$ flip-width of graphs in $\mathcal{C}$ is bounded for each $r\in \mathbb{N}$. For the precise definition, see Section~\ref{sec:prelim}. 

Flip-width generalizes twin-width~\cite{BonnetKTW2022}, which has been recently focused~\cite{BonnetKTW2022, BonnetGKTW2022, BonnetGKTW2021, BonnetGOSTT2022, BonnetGOT2023, BonnetKRT2022, PilipczukS2023, BonnetD2023, AhnHKO2022}. Bonnet, Kim, Reinald, and Thomass\'{e}~\cite{BonnetKTW2022} showed that
classes of bounded rank-width and classes of $H$-minor free graphs have bounded twin-width, and FO model checking is fixed parameter tractable on graphs of bounded twin-width provided that a witness of small twin-width is given.  
Toru\'{n}czyk~\cite{flipwidth} showed that for every $r\in \mathbb{N}$, if the twin-width of a graph $G$ is $d$, then its radius-$r$ flip-width is at most $2^d\cdot d^{\mathcal{O}(r)}$. Thus, any class of bounded twin-width has bounded flip-width. 
Toru\'{n}czyk~\cite[XI.5]{flipwidth} proposed a problem of finding an analogue of the model-checking result of Bonnet et al.~\cite{BonnetKTW2022} for classes of bounded flip-width.  
Furthermore, he showed that bounded expansion classes also have bounded flip-width.

In this paper, we focus on characterizing graphs of radius-$r$ flip-width at most $2$.  
For each $r\in \mathbb{N}\cup \{\infty\}$,
radius-$r$ flip-width possesses an interesting property that for every graph $G$, $G$ and its complement have the same radius-$r$ flip-width. This is because we can apply any strategy for $G$ to the complement of $G$ as well. Also, the radius-$r$ flip-width of a graph does not increase when taking its induced subgraph. Thus, it is natural to consider problems characterizing graphs of small radius-$r$ flip-width, in terms of induced subgraphs and complements.
Toru\'{n}czyk~\cite{flipwidth} showed that  radius-$\infty$ flip-width is asymptotically equivalent to clique-width~\cite{CourcelleO2000} and rank-width~\cite{Oum05}, meaning that a class of graphs has bounded radius-$\infty$ flip-width if and only if it has bounded clique-width.
Note that clique-width or rank-width may change when taking the complement. For example, complete graphs on at least two vertices have clique-width $2$ and rank-width $1$  while graphs without edges have clique-width $1$ and rank-width $0$.

Observe that for any $r\in \mathbb{N}\cup \{\infty\}$, graphs of radius-$r$ flip-width $1$ are exactly complete graphs and graphs with no edges. If a graph $G$ is neither a complete graph nor a graph without edges, then $G$ contains three vertices $v,w,z$ where $vw$ is an edge and $wz$ is not an edge of $G$. Since we only allow a partition with one set, the only graphs that we can make are $G$ and its complement. But if the runner stays on $w$, then the flipper cannot catch the runner. Clearly, if $G$ is a complete graph or a graph without edges, then the flipper can make a graph into a graph without edges, and catch the runner. 

\begin{figure}
  \centering
  \begin{center}
    \tikzstyle{v}=[circle,draw,fill=black!50,inner sep=0pt,minimum width=4pt]
    \begin{tikzpicture}
      \draw (0,.5) node[v](v){};
      \draw (-1.5,-.5) node[v](v1){};
      \draw(-.5,-.5)node[v](v2){};
      \draw (.5,-.5) node[v](v3){};
      \draw (1.5,-.5) node[v](v4){};
      \draw(v)--(v1)--(v2)--(v3)--(v4)--(v);
      \draw(0,-.5) node [label=below:$C_5$]{};
    \end{tikzpicture}
    $\quad\quad$ \begin{tikzpicture}
      \draw (0,.5) node[v](v){};
      \draw (-1.5,-.5) node[v](v1){};
      \draw(-.5,-.5)node[v](v2){};
      \draw (.5,-.5) node[v](v3){};
      \draw (1.5,-.5) node[v](v4){};
      \draw(v1)--(v2)--(v3)--(v4);
      \draw(v2)--(v)--(v3);
      \draw(0,-.5) node [label=below:bull]{};
    \end{tikzpicture}
    $\quad\quad$
    \begin{tikzpicture}
      \draw (0,.5) node[v](v){};
      \draw (-1.5,-.5) node[v](v1){};
      \draw(-.5,-.5)node[v](v2){};
      \draw (.5,-.5) node[v](v3){};
      \draw (1.5,-.5) node[v](v4){};
      \draw(v1)--(v2)--(v3)--(v4);
      \draw(v2)--(v)--(v3);
      \draw(v1)--(v)--(v4);
      \draw(0,-.5) node [label=below:gem]{};
    \end{tikzpicture}
    $\quad\quad$
    \begin{tikzpicture}
      \draw (0,.5) node[v](v){};
      \draw (-1.5,-.5) node[v](v1){};
      \draw(-.5,-.5)node[v](v2){};
      \draw (.5,-.5) node[v](v3){};
      \draw (1.5,-.5) node[v](v4){};
      \draw(v1)--(v2)--(v3)--(v4);
      \draw(0,-.5) node [label=below:co-gem]{};
    \end{tikzpicture}
  \end{center}
  \caption{$C_5$, bull, gem and co-gem graphs}
\label{fig:smallgraphs}
\end{figure}

We show that except the case $r=1$, the class of graphs of radius-$r$ flip-width at most $2$ is   exactly the class of $(C_5, \text{bull}, \text{gem}, \text{co-gem})$-free graphs, where $C_5$, $\text{bull}$, $\text{gem}$, and $\text{co-gem}$ are illustrated in Figure~\ref{fig:smallgraphs}. This class was known as the class of graphs totally decomposable with respect to bi-joins~\cite{bijoindecomposition}, and we use this structural result to prove the theorem. 

\begin{theorem}\label{thm:main}
  Let $r\in (\mathbb{N}\setminus \{1\})\cup \{\infty\}$. The class of graphs of radius-$r$ flip-width at most $2$ is the class of $(C_5, \text{bull}, \text{gem}, \text{co-gem})$-free graphs.
\end{theorem}

    We observe in Section~\ref{sec:conclusion} that gem and co-gem have radius-$1$ flip-width at most $2$. Thus, the class of graphs of radius-$1$ flip-width at most $2$ is strictly larger than the class of $(C_5, \text{bull}, \text{gem}, \text{co-gem})$-free graphs.  We leave the problem of characterizing graphs of radius-$1$ flip-width at most $2$ as an open problem.

    This paper is organized as follows. 
    In Section~\ref{sec:prelim}, we introduce basic definitions and notations, including  bi-joins and bi-join decompositions.
    In Section~\ref{sec:flipwidth3}, we prove that $C_5$, bull, gem and co-gem have radius-$r$ flip-width at least 3 for all $r\in (\mathbb{N}\setminus \{1\})\cup \{\infty\}$. In Section~\ref{sec:boundedflipwidth}, we prove that  $(C_5, \text{bull}, \text{gem}, \text{co-gem})$-free graphs have radius-$r$ flip-width  at most $2$ for all $r\in \mathbb{N}\cup \{\infty\}$.

\section{Preliminaries}\label{sec:prelim}

In this paper, all graphs are finite, undirected, and simple. Let $\mathbb{N}$ denote the set of all positive integers. For an integer $n$, let $[n]$ denote the set of all positive integers at most $n$.

Let $G$ be a graph. We denote by $V(G)$ the vertex set and $E(G)$ the edge set of $G$. For a set $S \subseteq V(G)$, let $G[S]$ denote the subgraph of $G$ induced by $S$. For $v\in V(G)$, let $N_G(v) := \{u \in V(G) : uv \in E(G)\}$ be the \emph{neighborhood} of $v$ in $G$. For $S\subseteq V(G)$, let $N_G(S)$ be the set of vertices in $V(G)\setminus S$ having a neighbor in $S$. A vertex $v$ in $G$ is a \emph{dominating vertex} if $N_G(v)\cup \{v\}=V(G)$.
For a graph $H$, we denote by $\overline{H}$ the complement of $H$. 

For two disjoint sets $A$ and $B$ of vertices in $G$, we say that $A$ is \emph{complete} to $B$ if for all $a\in A$ and $b\in B$, $a$ is adjacent to $b$, and it is \emph{anti-complete} to $B$ if for all $a\in A$ and $b\in B$, $a$ is not adjacent to $b$.

For a set $\mathcal{H}$ of graphs, a graph $G$ is \emph{$\mathcal{H}$-free} if $G$ contains no induced subgraph isomorphic to $H$ for every $H\in\mathcal{H}$.

For a graph $G$ and two sets $A, B$ of vertices in $G$, let $G\oplus (A,B)$ be the graph on the vertex set $V(G)$ such that for distinct vertices $u$ and $v$ in $G$,
$uv\in E(G\oplus (A,B) )$ if and only if     \begin{itemize}
        \item $(u,v)\in (A\times B)\cup (B\times A)$ and $uv\notin E(G)$, or
        \item $(u,v)\notin (A\times B)\cup (B\times A)$ and $uv\in E(G)$.
    \end{itemize}
We say that $G\oplus (A, B)$ is the graph obtained from $G$ by flipping the pair $(A, B)$. For a collection $\mathcal{S}=\{(A_1, B_1), \ldots, (A_m, B_m)\}$ of pairs of sets of vertices in $G$, we define 
\[G\oplus \mathcal{S}=G\oplus(A_1, B_1)\oplus \cdots \oplus (A_m, B_m).\]
Note that this is well-defined, because the order of flipping pairs in $\mathcal{S}$ does not affect the resulting graph.

Let $\mathcal{F}_{\obs}=\{C_5, \text{ bull, gem, co-gem}\}$. We will use the following result in Section~\ref{sec:flipwidth3}. 
\begin{theorem}[Hertz~\cite{Hertz1999}]\label{Hertz1999}
    For every $G\in \mathcal{F}_{\obs}$ and every vertex partition $(A,B)$ of $G$, $G\oplus (A, B)$ is isomorphic to a graph in $\mathcal{F}_{\obs}$. 
\end{theorem}

\subsection{Flip-width}
We define the flipper game and radius-$r$ flip-width.
Let $G$ be a graph and let $\mathcal{P}$ be a partition of the vertex set of $G$. We say that $G'$ is a \emph{$\mathcal{P}$-flip} of $G$ if $G'=G\oplus \mathcal{S}$ with $\mathcal{S}=\{(A_i, B_i):i\in I\}$ such that for all $i\in I$, $A_i, B_i\in \mathcal{P}$. Since flips are involutive and commute with each other, we may assume that $\mathcal{S}$ contains at most ${|\mathcal{P}|+1\choose 2}$ pairs. We say that a $\mathcal{P}$-flip $G'$ of $G$ is a \emph{$k$-flip} if $\abs{\mathcal{P}}\le k$.

Let $r\in \mathbb{N}\cup \{\infty\}$ and $k\in \mathbb{N}$. The \emph{flipper game} with radius $r$ and width $k$ on a graph $G$ is played by two players, \emph{flipper} and \emph{runner}. At the beginning, set $G_0=G$ and the runner selects a starting vertex $v_0$ of $G$. In each $i$-th round for $i>0$, the flipper announces a $k$-flip $G_i$  of $G$ and the runner knows $G_i$ and selects a new position $v_i\in V(G)$ following a path of length at most $r$ from $v_{i-1}$ in the previous graph $G_{i-1}$. The game is \emph{won by the flipper} if the runner's new position $v_i$ is isolated in $G_i$. In this case, we also say that the runner is caught in $G_i$. The alternating sequence $(v_0, G_1, v_1, G_2, \ldots )$ obtained by playing the flipper game will be called a \emph{play} for the flipper game.

The \emph{radius-$r$ flip-width} of a graph $G$, denoted by $\fw_r(G)$, is the minimum $k$ such that the flipper has a winning strategy in the flipper game of radius $r$ and width $k$ on $G$.

\subsection{Bi-joins and bi-join decompositions}\label{sec:bijoin}

A \emph{bipartition} of a set $V$ is an unordered pair $\{V_1, V_2\}$ such that $V_1 \neq \emptyset$, $V_2 \neq \emptyset$, $V_1 \cap V_2 = \emptyset$ and $V_1 \cup V_2 = V$. Two bipartitions $\{V_1, V_2\}$ and $\{W_1, W_2\}$ of a set $V$ \emph{overlap}  if $V_1\cap W_1$, $V_1\cap W_2$, $V_2\cap W_1$, and $V_2\cap W_2$ are non-empty. 
An \emph{ordered bipartition} of a set $V$ is an ordered pair $(V_1, V_2)$ such that $\{V_1, V_2\}$ is a bipartition of $V$. We define the concept of overlapping bipartitions for ordered bipartitions in the same manner.

Let $G$ be a graph. An ordered bipartition $(V_1, V_2)$ of $V(G)$ is a \emph{bi-join} in $G$ if there exist $W_1\subseteq V_1$ and $W_2\subseteq V_2$ such that 
\begin{itemize}
    \item $W_1$ is complete to $W_2$ and anti-complete to $V_2\setminus W_2$, 
    \item $V_1\setminus W_1$ is complete to $V_2\setminus W_2$ and anti-complete to $W_2$.
\end{itemize}
We say that $(W_1, V_1\setminus W_1, W_2, V_2\setminus W_2)$ is a \emph{bi-join partition} of $(V_1, V_2)$. One can observe that if~$(W_1, V_1\setminus W_1, W_2, V_2\setminus W_2)$ is a bi-join partition of $(V_1, V_2)$, then $(V_1\setminus W_1, W_1, V_2\setminus W_2, W_2)$ is also a bi-join partition of $(V_1, V_2)$. We say that $(W_1, V_1\setminus W_1, W_2, V_2\setminus W_2)$ and $(V_1\setminus W_1, W_1, V_2\setminus W_2, W_2)$ are equivalent bi-join partitions.

A bi-join $(V_1, V_2)$ in a graph $G$ is \emph{strong} if there is no bi-join $(W_1, W_2)$ in $G$ such that $(V_1, V_2)$ and $(W_1, W_2)$ overlap.

 A collection $\mathcal{B}$ of bipartitions of a set $V$ is called \emph{bipartitive} if 
\begin{itemize}
    \item for every $v\in V$, $\{\{v\}, V\setminus \{v\}\}\in \mathcal{B}$, and
    \item for every overlapping bipartitions $\{A, B\}, \{C,D\}\in \mathcal{B}$, the following bipartitions
    \begin{itemize}
        \item $\{A\cup C, B\cap D\}$, 
        \item $\{A\cup D, B\cap C\}$, 
        \item $\{B\cup C, A\cap D\}$, 
        \item $\{B\cup D, A\cap C\}$,
        \item $\{(A\cup C)\setminus (A\cap C), (A\cup D)\setminus (A\cap D)\}$, 
    \end{itemize} 
    are contained in $\mathcal{B}$.
\end{itemize}
Cunningham~\cite{Cunningham1974} developed a decomposition scheme for bipartitive families.
Using a decomposition scheme for bipartitive families, de Montgolfier and Rao~\cite{bijoindecomposition} presented the following decomposition theorem with respect to bi-joins.

Let $G$ be a graph. A pair $(T, \beta)$ of a tree $T$ and a bijection $\beta$ from $V(G)$ to the set of leaves of $T$ is called a \emph{decomposition} of $G$. For every $e=uv\in E(T)$, let $T_{e,u}$ and $T_{e,v}$ be the components of $T-e$ containing $u$ and $v$, respectively, and for each $z\in \{u,v\}$, let $A_{e,z}$ be the set of vertices in $G$ that are mapped to leaves contained in $T_{e,z}$ by $\beta$. 

    \begin{theorem}[Cunningham~\cite{Cunningham1974}; de Montgolfier and Rao~\cite{bijoindecomposition}]\label{thm:complete}
    For a graph $G$, there is a unique decomposition $(T, \beta)$ of $G$ such that 
    for every $e=uv\in E(G)$, $(A_{e,u}, A_{e,v})$ is a strong bi-join. Furthermore, the nodes of $T$ can be labeled complete or prime such that 
    \begin{itemize}
        \item if $t$ is a complete node with neighbors $t_1, \ldots t_k$, then for every $I\subseteq [k]$ with $1\le \abs{I}<k$, $\left(\bigcup\limits_{i \in I} A_{tt_i, t_i}, \bigcup\limits_{i \in [k]\setminus I} A_{tt_i, t_i}\right)$ is a bi-join, 
        \item if $t$ is a prime node with neighbors $t_1, \ldots, t_k$, then for every $a \in [k]$, $\left(A_{tt_a, t_a}, \bigcup\limits_{i \in [k]\setminus \{a\}} A_{tt_i, t_i}\right)$ is a bi-join, and
        \item every bi-join of $G$ is presented in the above way.
    \end{itemize}
\end{theorem}
The unique decomposition presented in Theorem~\ref{thm:complete} will be called the \emph{bi-join decomposition} of a given graph.

A bi-join $(V_1, V_2)$ in a graph $G$ is \emph{trivial} if $|V_1| = 1$ or $|V_2| = 1$.
We say that a graph $G$ is \emph{completely decomposable} with respect to bi-joins if every induced subgraph of $G$ on at least $4$ vertices has a non-trivial bi-join. de Montgolfier and Rao~\cite{bijoindecomposition} characterized graphs that are completely decomposable with respect to bi-joins.

\begin{theorem}[de Montgolfier and Rao~\cite{bijoindecomposition}]\label{thm:decomposable}
Let $G$ be a graph. The following are equivalent.
\begin{enumerate}
    \item $G$ is completely decomposable with respect to bi-joins.
    \item Every node of the bi-join decomposition of $G$ is complete.
    \item $G$ is ($C_5$, bull, gem, co-gem)-free. 
\end{enumerate}
\end{theorem}

\section{$C_5$, bull, gem, and co-gem have radius-$2$ flip-width at least $3$}\label{sec:flipwidth3}

In this section, we show that $C_5$, bull, gem and co-gem have radius-$r$ flip-width at least 3 for each $r\in (\mathbb{N}\setminus \{1\})\cup \{\infty\}$. Recall that $\mathcal{F}_{\obs}=\{C_5, \text{ bull, gem, co-gem}\}$.

We prove this in the following way. It is sufficient to prove that the graphs in $\mathcal{F}_{\obs}$ have radius-$2$ flip-width at least $3$. Suppose $G\in \mathcal{F}_{\obs}$ and $G$ has radius-$2$ flip-width at most $2$. Assume that $(v_0, G_1, v_1, \ldots, v_{n-1}, G_n)$ is a play for the flipper game of radius $2$ and width $2$ played on~$G$ such that 
\begin{itemize}
    \item for every possible choice of the runner in $G_n$, the flipper catches the runner at $G_n$.
\end{itemize}
In Lemma~\ref{lem:lastconfiguration2}, we obtain that the set $S$ of isolated vertices of $G_n$ has at least $2$ vertices. 
Using the fact that $G_n$ is a $2$-flip of $G$, we deduce that $\abs{S}=2$. In Lemma~\ref{lem:lastconfiguration2}, we further show that in such a case, $G_{n-1}$ has a connected component on the set $S$, and later we will verify that this is not possible.

\begin{lemma}\label{lem:lastconfiguration2}
    Let $G$ be a graph. 
    Let $(v_0, G_1, v_1, \ldots, v_{n-1}, G_n)$  with $n\ge 1$ be a play for the flipper game of radius $2$ and width $2$ such that for every possible choice of the runner in $G_n$, the flipper catches the runner at $G_n$. The following hold.
    \begin{itemize}
        \item The set $S$ of all isolated vertices of $G_n$ has at least two vertices.
        \item If $|S|=2$, then $G_{n-1}[S]$ is a connected component of $G_{n-1}$.
    \end{itemize}
\end{lemma}
\begin{proof}
        Since $n\ge 1$, the runner is not placed on an isolated vertex of $G=G_0$. Assume that the runner is caught at $v$ in $G_n$. Let $S$ be the set of all isolated vertices of $G_n$.
    
        Let $C$ be a connected component of $G_{n-1}$ containing $v$. By the rule of the game, the runner is in $C$ in $G_{n-1}$. If $\abs{V(C)} = 1$, then $V(C) = \{v\}$ and the runner is already caught at vertex $v$ in $G_{n-1}$. This contradicts the assumption that the runner is caught in $G_n$. So, $C$ has at least two vertices. 

    If $v$ is the unique isolated vertex of $G_n$, then the runner can move to a vertex in $V(C)\setminus \{v\}$ so that it is not caught, contradicting the assumption that the flipper catches the runner in $G_n$ for every possible choice of the runner.
    Thus, $S$ has at least $2$ vertices contained in $C$.

    For the second statement, assume that $|S|=2$ and $V(C)\setminus S\neq \emptyset$. Since this game is of radius $2$, 
    regardless of the position of $v_{n-1}$ in $G_{n-1}$, the runner can move to a vertex of $V(C)\setminus S$. Therefore, $V(C)\subseteq S$. Since $C$ has at least two vertices, we have $V(C)=S$.
\end{proof}

\begin{figure}
    \centering
        \tikzstyle{v}=[circle,draw,fill=black!30,inner sep=0pt,minimum width=4pt]
        \begin{tikzpicture}[scale = .8]
            \draw (-6.5,.5) node[v](v){};
            \node at (-6.5,.8) {\footnotesize A};
            \draw (-8,-.5) node[v](v1){};
            \node at (-8,-0.2) {\footnotesize A};
            \draw(-7,-.5)node[v](v2){};
            \node at (-7,-0.2) {\footnotesize B};
            \draw (-6,-.5) node[v](v3){};
            \node at (-6,-0.2) {\footnotesize B};
            \draw (-5,-.5) node[v](v4){};
            \node at (-5,-0.2) {\footnotesize B};
            \draw[thick] (v1)--(v2)--(v3)--(v4);
            \node at (-6.5, -1.2) {$G$};
            \draw[->] (-4.5,0) -- (-2.5,0);
            \node at (-3.5, .3) {\footnotesize $\oplus (B, B)$};
            \draw (-.5,.5) node[v](v){};
            \node at (-.5,.8) {\footnotesize A};
            \draw (-2,-.5) node[v](v1){};
            \node at (-2,-0.2) {\footnotesize A};
            \draw(-1,-.5)node[v](v2){};
            \node at (-1,-0.2) {\footnotesize B};
            \draw (0,-.5) node[v](v3){};
            \node at (0,-0.2) {\footnotesize B};
            \draw (1,-.5) node[v](v4){};
            \node at (1,-0.2) {\footnotesize B};
            \draw[thick] (v1)--(v2);
            \draw[thick] (v2) ..controls (-.25, -.75) and (.25, -.75) .. (v4);
            \node at (-.5, -1.2) {$G \oplus (B, B)$};
        \end{tikzpicture}
        \vskip 0.5cm
    
        \begin{tikzpicture}[scale = .8]
            \draw (-6.5,.5) node[v](v){};
            \node at (-6.5,.8) {\footnotesize B};
            \draw (-8,-.5) node[v](v1){};
            \node at (-8,-0.2) {\footnotesize A};
            \draw(-7,-.5)node[v](v2){};
            \node at (-7,-0.2) {\footnotesize A};
            \draw (-6,-.5) node[v](v3){};
            \node at (-6,-0.2) {\footnotesize B};
            \draw (-5,-.5) node[v](v4){};
            \node at (-5,-0.2) {\footnotesize B};
            \draw[thick] (v1)--(v2)--(v3)--(v4);
            \node at (-6.5, -1.2) {$G$};
            \draw[->] (-4.5,0) -- (-2.5,0);
            \node at (-3.5, .3) {\footnotesize $\oplus (A, A)$};
            \draw (-.5,.5) node[v](v){};
            \node at (-.5,.8) {\footnotesize B};
            \draw (-2,-.5) node[v](v1){};
            \node at (-2,-0.2) {\footnotesize A};
            \draw(-1,-.5)node[v](v2){};
            \node at (-1,-0.2) {\footnotesize A};
            \draw (0,-.5) node[v](v3){};
            \node at (0,-0.2) {\footnotesize B};
            \draw (1,-.5) node[v](v4){};
            \node at (1,-0.2) {\footnotesize B};
            \draw[thick] (v2)--(v3)--(v4);
            \node at (-.5, -1.2) {$G \oplus (A, A)$};
        \end{tikzpicture}
        
    \caption{Two possible cases in Lemma~\ref{lem:twoisolated}, where $G \oplus \mathcal{P}$ has at least two isolated vertices. }
    \label{fig:twoisolated}
\end{figure}

   In Lemmas~\ref{lem:twoisolated} and~\ref{lem:twoisolated2}, 
   we investigate which flips of graphs in $\mathcal{F}_{\obs}$ have at least two isolated vertices. We first deal with cases when we do not flip $(A,B)$ for a given vertex partition $(A,B)$ in Lemma~\ref{lem:twoisolated}, and then deal with cases when we flip $(A,B)$ in Lemma~\ref{lem:twoisolated2}. Note that for a graph $G\in \mathcal{F}_{\obs}$ and a vertex partition $(A,B)$ of $G$, by Theorem~\ref{Hertz1999}, $G\oplus (A,B)$ is isomorphic to a graph in $\mathcal{F}_{\obs}$. Because of it, Lemma~\ref{lem:twoisolated2} easily follows from Lemma~\ref{lem:twoisolated}.
\begin{lemma}\label{lem:twoisolated}
    Let $G\in \mathcal{F}_{\obs}$ and
    let $(A, B)$ be a vertex partition of $G$ with $|A|< |B|$ and let $H=G\oplus \mathcal{P}$ for some $\mathcal{P}\subseteq \{(A, A), (B, B)\}$. If $H$ has at least two isolated vertices, then either 
    \begin{itemize}
        \item $G$ is the co-gem and $A$ consists of the isolated vertex and one of the vertices of degree $1$ and $\mathcal{P}=\{(B, B)\}$, or
        \item $G$ is the co-gem and $A$ consists of one of the vertices of degree $1$ and its neighbor and $\mathcal{P}=\{(A, A)\}$.
    \end{itemize}   
    Moreover, in these cases, $H$ has exactly two isolated vertices.
\end{lemma}
\begin{proof}
Suppose $H$ has at least two isolated vertices.  Let $S\subseteq V(H)$ be a set of two isolated vertices in $H$. As $G$ has no two isolated vertices, $\mathcal{P}\neq \emptyset$. Observe that for every $D\in \mathcal{F}_{\obs}$, we have  $\overline{D}\in \mathcal{F}_{\obs}$. Thus, $A\neq \emptyset$. 

First assume that $|A|=1$ and let $v\in A$.  We may assume that $(A, A)\notin \mathcal{P}$, as flipping $(A,A)$ does not change the graph. Thus, $(B, B)\in \mathcal{P}$, and $B\cap S$ and $B\setminus S$ are complete in $G$. 

If $v\in S$, then $v$ is an isolated vertex in $G$ as $(A,B)\notin \mathcal{P}$. Thus $G$ is isomorphic to the co-gem. Since $\abs{S}=2$, $B$ contains exactly one vertex of $S$. Since $(B,B)\in \mathcal{P}$ and $B\cap S$ and $B\setminus S$ are complete in $G$, the vertex in $S\cap B$ has degree $3$ in $G$, a contradiction. If $v\notin S$, then $\abs{B\cap S}=\abs{B\setminus S}=2$ and $G$ contains a cycle of length $4$. Thus, $G$ is isomorphic to the gem. As the gem has no vertices of degree at most $1$, $v$ is complete to $B\setminus S$ in $G$. Then $G$ contains a subgraph isomorphic to $K_{2,3}$, a contradiction.

Now, we assume that $\abs{A}=2$.
Since $G$ has no two isolated vertices and has no component on two vertices, there is an edge between $A$ and $B$ in $G$. Since $(A,B)\notin \mathcal{P}$, one of the vertices in $A$ is not an isolated vertex in $H$. Thus, $\abs{S\cap A}\le 1$.

Assume that $\abs{S\cap A}=0$. Since $(A,B)\notin \mathcal{P}$, $A$ is anti-complete to $S$ in $G$. So, there is an edge between $S$ and $B\setminus S$ in $G$, and we have $(B, B)\in \mathcal{P}$. Then the two vertices in $S$ are vertices having the same neighborhood in $G$, which is not possible. 

Now, assume that $\abs{S\cap A}=1$. Let $w\in S\cap A$ and $z\in S\cap B$. Since $(A,B)\notin \mathcal{P}$, $w$ is anti-complete to $B$ and $z$ is anti-complete to $A$ in $G$.

If $(B, B)\in \mathcal{P}$, then $z$ is complete to $B\setminus \{z\}$ in $G$ and $z$ has degree $2$ in $G$. Since $w$ is anti-complete to $B$, $w$ has degree at most $1$ in $G$. So, $G$ is isomorphic to the bull or the co-gem. Since $w$ is anti-complete to the neighborhood of $z$ in $G$, $G$ is isomorphic to the co-gem and $w$ is the isolated vertex of $G$. This is one of the desired cases, depicted in Figure~\ref{fig:twoisolated}.

If $(B, B)\notin \mathcal{P}$, then $\mathcal{P}=\{(A,A)\}$. Since $(B, B)\notin \mathcal{P}$, $z$ is an isolated vertex in $G$, and $G$ is isomorphic to the co-gem. 
Since $w$ is anti-complete to $B$, $A$ consists of a vertex of degree $1$ and its neighbor in the co-gem. This is one of the desired cases, depicted in Figure~\ref{fig:twoisolated}.
\end{proof}

\begin{lemma}\label{lem:twoisolated2}
    Let $G\in \mathcal{F}_{\obs}$ and
    let $(A, B)$ be a vertex partition of $G$ with $|A|< |B|$ and let $H=G\oplus \mathcal{P}$ for some $(A, B)\in \mathcal{P}\subseteq \{(A, A), (B, B), (A, B)\}$. If $H$ has at least two isolated vertices, then either 
    \begin{itemize}
       \item $G$ is the gem and $A$ consists of two non-adjacent vertices of degree $2$ and $3$ and $\mathcal{P}=\{(B, B), (A, B)\}$, or
        \item $G$ is the gem and $A$ consists of the dominating vertex and a vertex of degree $3$ and $\mathcal{P}=\{(A, A), (A, B)\}$.
    \end{itemize}   
    Moreover, in these cases, $H$ has exactly two isolated vertices.
\end{lemma}
\begin{proof}
    Let $G'= G\oplus (A,B)$ and let $\mathcal{P}'=\mathcal{P}\setminus \{(A, B)\}$. Then $H=G'\oplus \mathcal{P}'$ where $\mathcal{P}'\subseteq \{(A,A), (B,B)\}$. By Theorem~\ref{Hertz1999}, $G'$ is isomorphic to a graph in $\mathcal{F}_{\obs}$. Thus, by Lemma~\ref{lem:twoisolated}, 
    either 
    \begin{itemize}
        \item $G'$ is isomorphic to the co-gem and $A$ consists of the isolated vertex and one of the vertices of degree $1$ and $\mathcal{P}'=\{(B, B)\}$, or
        \item $G'$ is isomorphic to the co-gem and $A$ consists of one of the vertices of degree $1$ and its neighbor and $\mathcal{P}'=\{(A, A)\}$. 
    \end{itemize}   
    Moreover, $H$ has exactly two isolated vertices. This implies the lemma.
\end{proof}

\begin{figure}
    \centering
        \tikzstyle{v}=[circle,draw,fill=black!30,inner sep=0pt,minimum width=4pt]
        \begin{tikzpicture}[scale = .8]
            \draw (-6.5,.5) node[v](v){};
            \node at (-6.5,.8) {\footnotesize B};
            \draw (-8,-.5) node[v](v1){};
            \node at (-8,-0.2) {\footnotesize A};
            \draw(-7,-.5)node[v](v2){};
            \node at (-7.1,-0.2) {\footnotesize B};
            \draw (-6,-.5) node[v](v3){};
            \node at (-5.9,-0.2) {\footnotesize B};
            \draw (-5,-.5) node[v](v4){};
            \node at (-5,-0.2) {\footnotesize A};
            \draw (v1)--(v2)--(v3)--(v4);
            \draw (v)--(v2);
            \draw (v)--(v3);
            \node at (-6.5, -1.2) {$G$};
            \draw[->] (-4.5,0) -- (-2.5,0);
            \node at (-3.5, .3) {\footnotesize $\oplus (B, B)$};
            \draw (-0.5,.5) node[v](v){};
            \node at (-0.5,.8) {\footnotesize B};
            \draw (-2,-.5) node[v](v1){};
            \node at (-2,-0.2) {\footnotesize A};
            \draw(-1,-.5)node[v](v2){};
            \node at (-1,-0.2) {\footnotesize B};
            \draw (0,-.5) node[v](v3){};
            \node at (0,-0.2) {\footnotesize B};
            \draw (1.0,-.5) node[v](v4){};
            \node at (1.0,-0.2) {\footnotesize A};
            \draw (v1)--(v2);
            \draw (v3)--(v4);
            \node at (-.5, -1.2) {$G \oplus (B, B)$};
          \end{tikzpicture}
        \vskip 0.5cm
    
        \begin{tikzpicture}[scale = .8]
            \draw (-6.5,.5) node[v](v){};
            \node at (-6.5,.8) {\footnotesize B};
            \draw (-8,-.5) node[v](v1){};
            \node at (-8,-0.2) {\footnotesize B};
            \draw(-7,-.5)node[v](v2){};
            \node at (-7,-0.2) {\footnotesize A};
            \draw (-6,-.5) node[v](v3){};
            \node at (-6,-0.2) {\footnotesize A};
            \draw (-5,-.5) node[v](v4){};
            \node at (-5,-0.2) {\footnotesize B};
            \draw (v1)--(v2)--(v3)--(v4);
            \node at (-6.5, -1.2) {$G$};
            \draw[->] (-4.5,0) -- (-2.5,0);
            \node at (-3.5, .3) {\footnotesize $\oplus (A, A)$};
            \draw (-.5,.5) node[v](v){};
            \node at (-.5,.8) {\footnotesize B};
            \draw (-2,-.5) node[v](v1){};
            \node at (-2,-0.2) {\footnotesize B};
            \draw(-1,-.5)node[v](v2){};
            \node at (-1,-0.2) {\footnotesize A};
            \draw (0,-.5) node[v](v3){};
            \node at (0,-0.2) {\footnotesize A};
            \draw (1,-.5) node[v](v4){};
            \node at (1,-0.2) {\footnotesize B};
            \draw (v1)--(v2);
            \draw (v3)--(v4);
            \node at (-.5, -1.2) {$G \oplus (A, A)$};
            
        \end{tikzpicture}
        
    \caption{Two possible cases in Lemma~\ref{lem:componentsizetwo} where $G \oplus \mathcal{P}$ has a connected component on exactly $2$ vertices.}
    \label{fig:componentsizetwo}
\end{figure}

   In Lemmas~\ref{lem:componentsizetwo} and~\ref{lem:componentsizetwo2}, 
   we investigate which flips of graphs in $\mathcal{F}_{\obs}$ have  a component on exactly two vertices. Similarly, we first deal with cases when we do not flip $(A,B)$ for a given vertex partition $(A,B)$ in Lemma~\ref{lem:componentsizetwo}, and then deal with cases when we flip $(A,B)$ in Lemma~\ref{lem:componentsizetwo2}.
   
\begin{lemma}\label{lem:componentsizetwo}
    Let $G\in \mathcal{F}_{\obs}$  and
    let $(A, B)$ be a vertex partition of $G$ with $|A|< |B|$ and let $H=G\oplus \mathcal{P}$ for some $\mathcal{P}\subseteq \{(A, A), (B, B)\}$. If $H$ has a connected component on exactly two vertices, then either
    \begin{itemize}
        \item $G$ is the bull and $A$ consists of two vertices of degree $1$  and $\mathcal{P}=\{(B,B)\}$, or
        \item $G$ is the co-gem and $A$ consists of two vertices of degree $2$ and $\mathcal{P}=\{(A,A)\}$.
    \end{itemize} 
\end{lemma}
\begin{proof}
Suppose $H$ has a connected component $C$ on exactly two vertices. As $G$ has no connected component of exactly two vertices, $\mathcal{P}\neq \emptyset$. Observe that for every $D\in \mathcal{F}_{\obs}$, we have  $\overline{D}\in \mathcal{F}_{\obs}$. Thus, $A\neq \emptyset$.

First assume that $|A|=1$ and let $v\in A$.  We may assume that $(A, A)\notin \mathcal{P}$, as flipping $(A,A)$ does not change the graph. Thus, $(B, B)\in \mathcal{P}$, and $B\cap V(C)$ and $B\setminus V(C)$ are complete in $G$. 

If $v\in V(C)$, then the vertex in $B\cap V(C)$ dominates $G$, and thus $G$ is isomorphic to the gem. Since $(A,B)\notin \mathcal{P}$, $v$ is anti-complete to $B\setminus V(C)$. So, $v$ has degree $1$ in $G$, but the gem has no vertex of degree $1$, a contradiction.
If $v\notin V(C)$, then $G$ contains a cycle of length $4$ and it is isomorphic to the gem. As the gem has no vertices of degree at most $1$, $v$ is complete to $B\setminus V(C)$ in $G$. Then $G$ contains a subgraph isomorphic to $K_{2,3}$, a contradiction.

Now, we assume that $\abs{A}=2$.
Since $G$ has no two isolated vertices and has no component on two vertices, there is an edge between $A$ and $B$ in $G$. Since $(A,B)\notin \mathcal{P}$, $H[A]$ is not a connected component of $H$. Thus, $\abs{V(C)\cap A}\le 1$.

Assume that $\abs{V(C)\cap A}=0$. Since $(A,B)\notin \mathcal{P}$, $A$ is anti-complete to $V(C)$ in $G$. So, there is an edge between $V(C)$ and $B\setminus V(C)$ in $G$, and we have $(B, B)\in \mathcal{P}$. Then the two vertices in $V(C)$ are vertices of degree $1$ having the same neighborhood in $G$, which is not possible. 

Now, assume that $\abs{V(C)\cap A}=1$. Let $w\in V(C)\cap A$ and $z\in V(C)\cap B$. Since $(A,B)\notin \mathcal{P}$, $w$ is anti-complete to $B\setminus \{z\}$ and $z$ is anti-complete to $A\setminus \{w\}$ in $G$.

If $(B, B)\in \mathcal{P}$, then $z$ is complete to $B\setminus \{z\}$ in $G$. Since $w$ is anti-complete to $B\setminus \{z\}$, $G$ cannot be the gem, and it is isomorphic to the bull. If $(A,A)$ is also in $\mathcal{P}$, then $H$ is isomorphic to the co-gem. Thus, $(A, A)\notin \mathcal{P}$ and this is one of the desired cases, depicted in Figure~\ref{fig:componentsizetwo}.

If $(B, B)\notin \mathcal{P}$, then $\mathcal{P}=\{(A,A)\}$. Since $(B, B)\notin \mathcal{P}$, $z$ has degree $1$ in $G$, while $w$ has degree $2$ in $G$. So, $G$ is isomorphic to the co-gem, and this is one of the desired cases, depicted in Figure~\ref{fig:componentsizetwo}.
\end{proof}

\begin{lemma}\label{lem:componentsizetwo2}
    Let $G\in \mathcal{F}_{\obs}$  and
    let $(A, B)$ be a vertex partition of $G$ with $|A|< |B|$ and let $H=G\oplus \mathcal{P}$ for some $(A, B)\in \mathcal{P}\subseteq \{(A, A), (B, B), (A, B)\}$. If $H$ has a connected component on exactly two vertices, then either
    \begin{itemize}
        \item $G$ is the gem and $A$ consists of two vertices of degree $2$ and $\mathcal{P}=\{(B, B), (A, B)\}$, or
        \item $G$ is the bull and $A$ consists of two vertices of degree $3$ and $\mathcal{P}=\{(A, A), (A, B)\}$.
    \end{itemize}
\end{lemma}
\begin{proof}
    Let $G'= G\oplus (A,B)$ and let $\mathcal{P}'=\mathcal{P}\setminus \{(A, B)\}$. Then $H=G'\oplus \mathcal{P}'$ where $\mathcal{P}'\subseteq \{(A,A), (B,B)\}$. By Theorem~\ref{Hertz1999}, $G'$ is isomorphic to a graph in $\mathcal{F}_{\obs}$. Thus, by Lemma~\ref{lem:componentsizetwo}, 
    either 
  \begin{itemize}
        \item $G'$ is isomorphic to the bull and $A$ consists of two vertices of degree $1$  and $\mathcal{P}'=\{(B,B)\}$, or
        \item $G'$ is isomorphic to the co-gem and $A$ consists of two vertices of degree $2$ and $\mathcal{P}'=\{(A,A)\}$.
    \end{itemize}  
    This implies the lemma.
\end{proof}
Now, we prove our main result.

\begin{theorem}\label{thm:main1}
    For every $G\in \mathcal{F}_{\obs}$ and every $r\in (\mathbb{N}\setminus \{1\})\cup \{\infty\}$, $\fw_r(G) \ge 3$.
\end{theorem}
\begin{proof}
Let $G\in \mathcal{F}_{\obs}$. It is sufficient to show that $\fw_2(G) \ge 3$ as $\fw_r(G)\ge \fw_2(G)$ for all $r\in (\mathbb{N}\setminus \{1\})\cup \{\infty\}$.

Suppose that $\fw_2(G) \le 2$. Then there is a play $(v_0, G_1, v_1, \ldots, v_{n-1}, G_n)$ for the flipper game of radius $2$ and width $2$ played on $G$ such that 
for every possible choice of the runner in $G_n$, the flipper catches the runner at $G_n$.

Let $S$ be the set of all isolated vertices in $G_n$.
By Lemma~\ref{lem:lastconfiguration2}, $S$ has at least two isolated vertices.
For each $i\in [n]$, let $(A_i, B_i)$ be a vertex partition of $G$, where $|A_i|< |B_i|$ and $A_i$ is possibly an empty set such that $G_i=G\oplus \mathcal{S}_i$ for some $\mathcal{S}_i\subseteq \{(A_i, A_i), (B_i, B_i), (A_i,B_i)\}$.  

By Lemmas~\ref{lem:twoisolated} and~\ref{lem:twoisolated2}, we have that either
    \begin{itemize}
        \item $G$ is the co-gem and $A_n$ consists of the isolated vertex and one of the vertices of degree $1$ and $\mathcal{S}_n=\{(B_n, B_n)\}$, 
        \item $G$ is the co-gem and $A_n$ consists of one of the vertices of degree $1$ and its neighbor and $\mathcal{S}_n=\{(A_n, A_n)\}$, 
        \item $G$ is the gem and $A_n$ consists of two non-adjacent vertices of degree $2$ and $3$ and $\mathcal{S}_n=\{(B_n, B_n), (A_n, B_n)\}$, or
        \item $G$ is the gem and $A_n$ consists of the dominating vertex and a vertex of degree $3$ and $\mathcal{S}_n=\{(A_n, A_n), (A_n, B_n)\}$.
    \end{itemize}
    Moreover, in any case, we have that $\abs{S}=2$.
    By Lemma~\ref{lem:lastconfiguration2}, $G_{n-1}[S]$ is a connected component of $G_{n-1}$. Let $C=G_{n-1}[S]$.

Now, by Lemmas~\ref{lem:componentsizetwo} and~\ref{lem:componentsizetwo2},
    we have that either
    \begin{itemize}
        \item $G$ is the bull and $A_{n - 1}$ consists of two vertices of degree $1$  and $\mathcal{S}_{n - 1}=\{(B_{n - 1}, B_{n - 1})\}$, 
        \item $G$ is the co-gem and $A_{n - 1}$ consists of two vertices of degree $2$ and $\mathcal{S}_{n - 1}=\{(A_{n - 1}, A_{n - 1})\}$, 
        \item $G$ is the gem and $A_{n - 1}$ consists of two vertices of degree $2$ and~$\mathcal{S}_{n - 1}=\{(B_{n - 1}, B_{n - 1}), (A_{n - 1}, B_{n - 1})\}$, or
        \item $G$ is the bull and $A_{n - 1}$ consists of two vertices of degree $3$ and $\mathcal{S}_{n - 1}=\{(A_{n - 1}, A_{n - 1}), (A_{n - 1}, B_{n - 1})\}$.
    \end{itemize}

     Combining the obtained results, we deduce that $G$ is either the co-gem or the gem. 
     
     Assume that $G$ is the co-gem. Then $S$ contains the isolated vertex of $G$, while $C$ does not contain the isolated vertex. So, it contradicts the fact that $V(C)=S$.

     Assume that $G$ is the gem. Then $S$ contains the dominating vertex of $G$, while $C$ does not contain the dominating vertex. This contradicts the fact that $V(C)=S$.

     Therefore, such a play does not exist. We conclude that $\fw_2(G)\ge 3$.
     \end{proof}

\section{Radius-$r$ flip-width of ($C_5$, bull, gem, co-gem)-free graphs}\label{sec:boundedflipwidth}

In this section, we prove that ($C_5$, bull, gem, co-gem)-free graphs have radius-$r$ flip-width at most~$2$ for all $r\in \mathbb{N}\cup \{\infty\}$.

 Let $G$ be a graph. A partition $(A, B, C)$ of $V(G)$ is a \emph{triple of bi-joins} if $(A, B\cup C)$, $(B, A\cup C)$, and $(C, A\cup B)$ are bi-joins of $G$. The following lemma is useful to find a proper $2$-flip to remove edges incident with two distinct sets of $A$, $B$, and $C$.
\begin{lemma}\label{lem:threepartition}
    Let $G$ be a graph and let $(A, B, C)$ be a triple of bi-joins of $G$. Then there is a $2$-flip $G'$ of $G$ such that in $G'$ there are no edges incident with two distinct sets of $A$, $B$, and $C$.
\end{lemma}
\begin{proof}
    Let $(A_1, A_2)$ be the $A$-part of the bi-join partition of $(A, B\cup C)$, 
    let $(B_1, B_2)$ be the $B$-part of the bi-join partition of $(B, A\cup C)$, and
    let $(C_1, C_2)$ be the $C$-part of the bi-join partition of $(C, A\cup B)$. 
    
    Clearly, $(A_1, A_2, B_1, B_2)$ or $(A_1, A_2, B_2, B_1)$ is a bi-join partition of $(A, B)$ in $G[A\cup B]$.
    By exchanging $B_1$ and $B_2$ if necessary, we may assume that $(A_1, A_2, B_1, B_2)$ is a bi-join partition of $(A, B)$ in $G[A\cup B]$. 
    Similarly, by exchanging $C_1$ and $C_2$ if necessary, we may assume that $(A_1, A_2, C_1, C_2)$ is a bi-join partition of $(A, C)$ in $G[A\cup C]$.

    Now, there are two cases; either a bi-join partition of $(B, C)$ in $G[B\cup C]$ is equivalent to $(B_1, B_2, C_1, C_2)$, or a bi-join partition of $(B, C)$ in $G[B\cup C]$ is equivalent to $(B_1, B_2, C_2, C_1)$.  
    
    \medskip
    \noindent\textbf{(Case 1. A bi-join partition of $(B, C)$ in $G[B\cup C]$ is equivalent to $(B_1, B_2, C_1, C_2)$.)}

    In this case, the set of all edges incident with two sets of $A$, $B$, and $C$ in $G$ is the union of 
    \begin{itemize}
        \item $\{vw:v\in V\in \{A_1, B_1, C_1\}, w\in W\in \{A_1, B_1, C_1\}, V\neq W\}$, and
        \item $\{vw:v\in V\in \{A_2, B_2, C_2\}, w\in W\in \{A_2, B_2, C_2\}, V\neq W\}$.
    \end{itemize}
    Let $X=A_1\cup B_1\cup C_1$ and $Y=A_2\cup B_2\cup C_2$.
    Then in $G\oplus (X, X)\oplus (Y, Y)$, there are no edges incident with two sets of $A$, $B$, and $C$.

    \medskip
    \noindent\textbf{(Case 2. A bi-join partition of $(B, C)$ in $G[B\cup C]$ is equivalent to $(B_1, B_2, C_2, C_1)$.)}
      
      In this case, the set of all edges incident with two sets of $A$, $B$, and $C$ in $G$ is the union of 
    \begin{itemize}
        \item $\{vw:v\in A_1, w\in B_1\cup C_1\}$, 
        \item $\{vw:v\in B_2, w\in C_1\cup A_2\}$, and 
        \item $\{vw:v\in C_2, w\in A_2\cup B_1\}$. 
    \end{itemize}
     Let $X=A_1\cup B_2\cup C_2$ and $Y=A_2\cup B_1\cup C_1$.
    Then in $G\oplus (X, Y)$, there are no edges incident with two sets of $A$, $B$, and $C$.
    \end{proof}

\begin{theorem}\label{thm:bounding}
    Let $G$ be a ($C_5$, bull, gem, co-gem)-free graph and $r\in \mathbb{N}\cup \{\infty\}$. Then $\fw_r(G) \le 2$.
\end{theorem}
\begin{proof}
    It is sufficient to show that $\fw_\infty(G) \le 2$ as $\fw_r(G)\le \fw_{\infty}(G)$ for all $r\in \mathbb{N}$.
    
    We may assume that $\abs{V(G)}\ge 3$; otherwise, $G$ has radius-$\infty$ flip-width at most $1$.
    Let $(T, \beta)$ be the bi-join decomposition of $G$. By Theorem~\ref{thm:decomposable}, every node of $T$ is a complete node. Since $\abs{V(G)}\ge 3$, $T$ has at least one internal node. Let $z$ be an internal node of $T$, and consider it as the root of $T$.

    For every $t\in V(T)$, let $X_t$ be the set of vertices in $G$ that are mapped to a descendant of $t$ in $T$ by $\beta$, and let $Y_t=V(G)\setminus X_t$. 

    Let $d$ be the longest distance from a leaf to the root $z$. For every node $t$ of $T$, let $d(t)$ be the distance between $z$ and $t$. We prove by induction on $d-d(t)$ the following statement.
    \begin{itemize}
        \item[$(\ast)$] Let $G^*$ be a $2$-flip of $G$ where there are no edges between $X_t$ and $Y_t$ in $G^*$ and the runner is given in $X_t$. Then the flipper has a winning strategy to catch the runner. 
    \end{itemize}
    In the base case when $t$ is a leaf of $T$,
    by the assumption in $(\ast)$, $X_t$ is an isolated vertex in $G^*$. It is clear that in such a case the flipper can catch the runner.

    Thus, we may assume that $t$ is an internal node. Also, we may assume that $d-d(t)>0$ and the statement is true for all nodes $t'$ where $d-d(t')<d-d(t)$. Note that if the statement is true, then certainly, the flipper has a winning strategy for~$G$ by considering $t=z$.

    Let $t_1, \ldots, t_m$ be the children of $t$. If $m=1$, then 
    $X_t=X_{t_1}$ and $Y_t=Y_{t_1}$. Since $d-d(t_1)<d-d(t)$, by induction hypothesis, the flipper has a winning strategy. So, we may assume that $m\ge 2$.

    Observe that since $t$ is a complete node, for each $i\in [m-1]$,
    \[\mathcal{P}_i=\left(Y_t\cup \bigcup_{j\in [i]\setminus \{i\}} X_{t_j}, X_{t_i}, \bigcup_{j\in [m]\setminus [i]} X_{t_j} \right)\] is a triple of bi-joins of $G$. By Lemma~\ref{lem:threepartition}, there is a $2$-flip $G_i$ of $G$ such that in $G_i$, there are no edges incident with two sets of $\mathcal{P}_i$.

    First assume that $m=2$. We take the $2$-flip $G_1$. Then the runner is contained in $X_{t_1}$ or $X_{t_2}$, and there are no edges in $G_1$ incident with two distinct sets of $X_{t_1}$, $X_{t_2}$, and $Y_t$. Since $d-d(t_1)=d-d(t_2)<d-d(t)$, by the induction hypothesis, the flipper has a winning strategy for $G_1$. Thus we may assume that $m\ge 3$.

    From $i=1$ to $m-1$, we do the following. First we take the $2$-flip $G_1$. If the runner is contained in $X_{t_1}$, then 
    since there is no edge in $G_1$ incident with two distinct sets of $\mathcal{P}_1$, there are no edges between $X_{t_1}$ and $Y_{t_1}$ in $G_1$.
    So, by the induction hypothesis, the flipper has a winning strategy for $G_1$. Thus, we may assume that the runner is contained in  $\bigcup_{j\in [m]\setminus [1]} X_{t_j}$. 
    
    For $i<m-1$, assume that the runner is contained in $\bigcup_{j\in [m]\setminus [i]} X_{t_j}$. We take the $2$-flip $G_{i+1}$. Since in $G_i$, there are no edges between $\bigcup_{j\in [m]\setminus [i]} X_{t_j}$ and the rest, the runner is still contained in $\bigcup_{j\in [m]\setminus [i]} X_{t_j}$ in $G_{i+1}$. If the runner is contained in $X_{t_{i+1}}$, then by the induction hypothesis, the flipper has a winning strategy for $G_{i+1}$. Thus, we may assume that the runner is contained in 
    $\bigcup_{j\in [m]\setminus [i+1]} X_{t_j}$.
    This shows that at the end, the runner is contained in $X_{t_{m-1}}$ or $X_{t_m}$ in $G_{m-1}$, and we can apply the induction hypothesis to obtain a winning strategy for the flipper. 

    This proves the statement $(\ast)$. In particular, for $t=z$, $G$ itself is a $2$-flip of $G$ where there are no edges between $X_t=V(G)$ and $Y_t=\emptyset$. Therefore, the flipper has a winning strategy to catch the runner.
\end{proof}

Theorem~\ref{thm:main} follows from Theorems~\ref{thm:main1} and \ref{thm:bounding}.

\section{Conclusion}\label{sec:conclusion}

In this paper, we show that for all $r\in (\mathbb{N}\setminus \{1\})\cup \{\infty\}$, the class of graphs of radius-$r$ flip-width at most $2$ is the class of $(C_5, \text{bull}, \text{gem}, \text{co-gem})$-free graphs.

We observe that the radius-$1$ flip-width of the gem and the co-gem are at most $2$. This shows that the class of graphs of radius-$1$ flip-width at most $2$ is strictly larger than the class of graphs of radius-$2$ flip-width at most $2$.

\begin{theorem}
    The radius-$1$ flip-width of the gem and the co-gem are at most $2$.
\end{theorem}
\begin{proof}
    It is sufficient to show for co-gem. We give a strategy for the flipper.
    Let $\{w,a,b,c,d\}$ be the vertex set of the co-gem $G$ where $w$ is the isolated vertex and $abcd$ is an induced path in $G$.

    We take $(A_1, B_1)=(\{w,d\}, \{a,b,c\})$ and $G_1=G\oplus (B_1, B_1)$. We may assume that the runner is in $\{a, c, d\}$, otherwise, it is caught.

    We take $(A_2, B_2)=(\{w,a,b\}, \{c,d\})$ and $G_2=G\oplus (A_2, A_2)\oplus (B_2, B_2)$.
    Now, we may assume that the runner is in $\{a,c\}$.
    We divide cases.

    \medskip
    \noindent\textbf{(Case 1. The runner is in $a$ at $G_2$.)}
   
    We take $(A_3, B_3)=(\{a,b\}, \{w,c,d\})$ and $G_3=G\oplus (A_3, A_3)$.
    Since the runner can move to at most $1$ vertex in $G_2$, the runner is contained in $\{w,a\}$ in $G_3$, and it is caught.

   \medskip
    \noindent\textbf{(Case 2. The runner is in $c$ at $G_2$.)}

    We take $(A_3, B_3)=(\{w, a, c\}, \{b, d\})$ and $G_3=G\oplus (A_3, B_3)$. Then we may assume that the runner is moved to $b$ in $G_3$.

    Lastly, we take $(A_4, B_4)=(\{a,b,c\}, \{w,d\})$ and $G_4=G\oplus (A_4, A_4)$. Then the runner must be in $\{w,b\}$, and it is caught.
\end{proof}

    We leave the problem to characterize graphs of radius-$1$ flip-width at most $2$ as an open problem.

\paragraph*{Acknowledgements}
The authors would like to thank the anonymous referees for the careful reading of the manuscript and numerous suggestions that helped to improve the presentation.


\begin{thebibliography}{10}

\bibitem{AhnHKO2022}
Jungho Ahn, Kevin Hendrey, Donggyu Kim, and Sang-Il Oum.
\newblock Bounds for the twin-width of graphs.
\newblock {\em SIAM J. Discrete Math.}, 36(3):2352--2366, 2022.

\bibitem{BonnetD2023}
\'{E}douard Bonnet and Hugues D\'{e}pr\'{e}s.
\newblock Twin-width can be exponential in treewidth.
\newblock {\em J. Combin. Theory Ser. B}, 161:1--14, 2023.

\bibitem{BonnetGKTW2021}
\'{E}douard Bonnet, Colin Geniet, Eun~Jung Kim, St\'{e}phan Thomass\'{e}, and
  R\'{e}mi Watrigant.
\newblock Twin-width {III}: max independent set, min dominating set, and
  coloring.
\newblock In {\em 48th {I}nternational {C}olloquium on {A}utomata, {L}anguages,
  and {P}rogramming}, volume 198 of {\em LIPIcs. Leibniz Int. Proc. Inform.},
  pages Art. No. 35, 20. Schloss Dagstuhl. Leibniz-Zent. Inform., Wadern, 2021.

\bibitem{BonnetGKTW2022}
\'{E}douard Bonnet, Colin Geniet, Eun~Jung Kim, St\'{e}phan Thomass\'{e}, and
  R\'{e}mi Watrigant.
\newblock Twin-width {II}: small classes.
\newblock {\em Comb. Theory}, 2(2):Paper No. 10, 42, 2022.

\bibitem{BonnetGOSTT2022}
\'{E}douard Bonnet, Ugo Giocanti, Patrice Ossona~de Mendez, Pierre Simon,
  St\'{e}phan Thomass\'{e}, and Szymon Toru\'{n}czyk.
\newblock Twin-width {IV}: ordered graphs and matrices.
\newblock {\em J. ACM}, 71(3):Art. 21, 45, 2024.

\bibitem{BonnetGOT2023}
\'{E}douard Bonnet, Ugo Giocanti, Patrice Ossona~de Mendez, and St\'{e}phan
  Thomass\'{e}.
\newblock Twin-width {V}: linear minors, modular counting, and matrix
  multiplication.
\newblock In {\em 40th {I}nternational {S}ymposium on {T}heoretical {A}spects
  of {C}omputer {S}cience ({STACS} 2023)}, volume 254 of {\em LIPIcs. Leibniz
  Int. Proc. Inform.}, pages Paper No. 15, 16. Schloss Dagstuhl. Leibniz-Zent.
  Inform., Wadern, 2023.

\bibitem{BonnetKRT2022}
\'{E}douard Bonnet, Eun~Jung Kim, Amadeus Reinald, and St\'{e}phan
  Thomass\'{e}.
\newblock Twin-width {VI}: the lens of contraction sequences.
\newblock In {\em Proceedings of the 2022 {A}nnual {ACM}-{SIAM} {S}ymposium on
  {D}iscrete {A}lgorithms ({SODA})}, pages 1036--1056. [Society for Industrial
  and Applied Mathematics (SIAM)], Philadelphia, PA, 2022.

\bibitem{BonnetKTW2022}
\'{E}douard Bonnet, Eun~Jung Kim, St\'{e}phan Thomass\'{e}, and R\'{e}mi
  Watrigant.
\newblock Twin-width {I}: {T}ractable {FO} model checking.
\newblock {\em J. ACM}, 69(1):Art. 3, 46, 2022.

\bibitem{CourcelleO2000}
Bruno Courcelle and Stephan Olariu.
\newblock Upper bounds to the clique width of graphs.
\newblock {\em Discrete Appl. Math.}, 101(1-3):77--114, 2000.

\bibitem{Cunningham1974}
William~Harry Cunningham.
\newblock {\em {A combinatorial decomposition theory}}.
\newblock ProQuest LLC, Ann Arbor, MI, 1974.
\newblock Thesis (Ph.D.)--University of Waterloo (Canada).

\bibitem{bijoindecomposition}
Fabien de~Montgolfier and Michael Rao.
\newblock {Bipartitive families and the bi-join decomposition}.
\newblock submitted to special issue of Discrete Mathematics devoted to ICGT05,
  2005.

\bibitem{Hertz1999}
Alain Hertz.
\newblock On perfect switching classes.
\newblock {\em Discrete Appl. Math.}, 89(1-3):263--267, 1998.

\bibitem{Oum05}
Sang-il Oum.
\newblock Rank-width and vertex-minors.
\newblock {\em J. Combin. Theory Ser. B}, 95(1):79--100, 2005.

\bibitem{PilipczukS2023}
{Micha\l} Pilipczuk and Marek {Soko\l owski}.
\newblock Graphs of bounded twin-width are quasi-polynomially {$\chi$}-bounded.
\newblock {\em J. Combin. Theory Ser. B}, 161:382--406, 2023.

\bibitem{flipwidth}
Szymon Toru\'{n}czyk.
\newblock Flip-width: cops and robber on dense graphs.
\newblock In {\em 2023 {IEEE} 64th {A}nnual {S}ymposium on {F}oundations of
  {C}omputer {S}cience---{FOCS} 2023}, pages 663--700. IEEE Computer Soc., Los
  Alamitos, CA, [2023] \copyright2023.

\end{thebibliography}
\end{document}